\documentclass[12pt,a4paper]{article}
\usepackage[bottom=1.5in, top=1.2in]{geometry}
\usepackage{amssymb,amsmath,amsthm}
\usepackage[english]{babel}
\usepackage{t1enc}
\usepackage[utf8]{inputenc}
\usepackage{epsfig}
\usepackage{subfig}
\usepackage{epstopdf}
\usepackage{xcolor}
\usepackage{comment}
\usepackage{hyperref}
\usepackage{graphicx}

\newtheorem{thm}{Theorem}[section]
\newtheorem{cor}[thm]{Corollary}
\newtheorem{lem}[thm]{Lemma}
\newtheorem{prop}[thm]{Proposition}
\newtheorem{claim}[thm]{Claim}

\newtheorem{obs}[thm]{Observation}

\newtheorem{ques}[thm]{Question}

\theoremstyle{definition}
\newtheorem{defi}[thm]{Definition}

\usepackage{indentfirst}
\def\red1{\color{red}1\color{black}}
\def\blue1{\color{blue}1\color{black}}

\newcommand{\floor}[1]{\lfloor #1 \rfloor}
\newcommand{\ceil}[1]{\lceil #1 \rceil}
\newcommand{\orarrow}[1]{\overrightarrow{#1}}
\newcommand{\shortto}{\mathrel{\scalebox{0.75}[1]{$\scriptstyle\rightarrow$}}}

\def\01{0\text{-}1}

\makeatletter
\def\blfootnote{\gdef\@thefnmark{}\@footnotetext}
\makeatother

\title{Saturation and semisaturation of oriented graphs}
\author{Vladimir Bo\v skovi\'c \thanks{HUN-REN Alfréd Rényi Institute of Mathematics, Budapest, Hungary.\\Email address: boskovic at renyi dot hu} \and Bal\'azs Keszegh\thanks{HUN-REN Alfréd Rényi Institute of Mathematics and ELTE Eötvös Loránd University, Budapest, Hungary.\\Email address: keszegh at renyi dot hu\\}}

\begin{document}

\maketitle

\begin{abstract}
    We consider a natural extension of the classical saturation problem to oriented graphs, defined as the minimum number of edges in an oriented graph $H$ on $n$ vertices that avoids a fixed oriented graph $G$, but adding any missing edge to $H$ creates a copy of $G$.

    Our main result shows that saturation functions of oriented graphs attain at least four distinct orders of magnitude. In particular, if $G$ is acyclic, then its saturation function is bounded if and only if $G$ contains an isolated edge; otherwise, it is linear. Certain tournaments together with an isolated edge give further infinite examples of bounded and linear saturation functions. Two additional orders of magnitude are $\Theta(n \sqrt{\log n})$ and $\Theta(n \log n)$. For these two, the lower bounds are obtained by adapting a recently developed approach for the saturation problem in edge-ordered graphs, while the matching upper bounds are constructive.
    
    Finally, we study the semisaturation problem, in which we require only that adding any missing edge to $H$ creates a copy of $G$. We completely characterize the possible orders of magnitude: the semisaturation function is either bounded, linear, or $\Theta(n \log n)$. 
\end{abstract}

\section{Introduction}

An \emph{oriented graph} is a pair $G = (V, E)$ where $V$ is the vertex set, and $E$ is a set of oriented edges such that if $\orarrow{uv} \in E$, then $\orarrow{vu} \notin E$. We use the notation $\orarrow{uv}$ to describe an edge that goes from $u \in V$ to $v \in V$.

\emph{Extremal problems} are classical problems in combinatorics that deal with the maximal density that a combinatorial structure can have while avoiding some forbidden substructure. Such problems have a long history, and graphs are among the most studied objects in this area. \emph{Saturation problems}, on the other hand, deal with the minimal density that a combinatorial structure can have so that it avoids the forbidden substructure, but extending it in any way introduces the forbidden substructure. Saturation problems have been studied for many combinatorial structures, including graphs, hypergraphs, set systems, and $0$-$1$ matrices. More recently, we have studied saturation problems for vertex-ordered~\cite{boskesz023} and edge-ordered~\cite{boskesz026} graphs, which are closely related to the present work. In this paper, our goal is to systematically study the saturation problems for oriented graphs.

One of the fundamental questions about saturation functions is to determine the range of possible orders of magnitude. Compared to the classical extremal functions, they often appear to have simpler behaviour. In particular, for several of the settings mentioned above, it was shown that the saturation functions can either be bounded or linear (as a function of the size of the base structure, e.g., the number of vertices of a graph). For oriented graphs, however, that is no longer the case. We show that, in general, at least four distinct orders of magnitude can occur.

We first define the saturation and semisaturation functions. Throughout the paper, a \emph{copy} of $G$ in $H$ means a subgraph of $H$ that is isomorphic to $G$.

\begin{defi} [Saturation] \label{sat_func}
    Let $G$ be an oriented graph. We denote by $\mathrm{sat}_{\shortto}(n,G)$ the minimum number of edges in an oriented graph $H$ on $n$ vertices that does not contain a copy of $G$, but adding any new edge to $H$ introduces a copy of $G$.
\end{defi}

\begin{defi} [Semisaturation] \label{ssat_func}
    Let $G$ be an oriented graph. We denote by $\mathrm{ssat}_{\shortto}(n,G)$ the minimum number of edges in an oriented graph $H$ on $n$ vertices such that adding any new edge to $H$ introduces a new copy of $G$.
\end{defi}

The following observation relates the saturation and semisaturation functions.

\begin{obs} \label{ssat < sat}
    For every oriented graph $G$, we have
    \[\mathrm{ssat}_{\shortto}(n,G) \leq \mathrm{sat}_{\shortto}(n,G).\]
\end{obs}

Throughout the paper $``+"$ denotes the disjoint union of two graphs. Slightly abusing this notation, $G+v$ describes the graph obtained from $G$ by adding an isolated vertex $v$, while $G+e$ denotes the graph obtained by taking $G$ and two additional vertices with an oriented edge $e$ connecting them. Further, for $v \in V(G)$, $G-v$ denotes the subgraph of $G$ induced by the vertex set $V(G)\setminus \{v\}$.

The following result allows us to focus on the saturation and semisaturation functions of oriented graphs without isolated vertices for the remainder of the paper.

\begin{lem}
    Let $G$ be an oriented graph, and $v$ be an isolated vertex. If $n > |V(G)|$, then 
    \[\mathrm{sat}_{\shortto}(n, G) = \mathrm{sat}_{\shortto}(n, G + v) \text{ and } \mathrm{ssat}_{\shortto}(n, G) = \mathrm{ssat}_{\shortto}(n, G + v).\]
\end{lem}

\begin{proof}
    Let $n > |V(G)|$ and let $H_n$ be a graph that has $n$ vertices and $\mathrm{sat}_{\shortto}(n, G)$ edges, such that it avoids $G$, and adding any new edge to $H_n$ introduces a copy of $G$. If $H_n$ avoids $G$, then it also avoids $G + v$. Moreover, if we add an edge $u_1u_2$ to $H_n$, denote by $G'$ the copy of $G$. Then there exists $u \in V(H_n) \setminus V(G')$ such that $G' + u$ is a copy of $G + v$.
    
    This proves $\mathrm{sat}_{\shortto}(n, G) \geq \mathrm{sat}_{\shortto}(n, G + v)$ and $\mathrm{ssat}_{\shortto}(n, G) \geq \mathrm{ssat}_{\shortto}(n, G + v)$. We can repeat the same argument for the opposite direction.
\end{proof}

For oriented graphs with an isolated edge, we show a dichotomy for the saturation functions. However, in this case, we cannot characterize the respective families. The following proof is a standard argument, as it is very similar to the respective proofs about graphs, vertex ordered graphs, etc.

\begin{lem} \label{dichotomy_isolated_edge}
    For every oriented graph $G$ that has an isolated edge, we have
    \[\mathrm{sat}_{\shortto}(n, G) = O(1) \text{ or } \Theta(n).\]
    If $G$ does not contain isolated edges, then 
    \[\mathrm{sat}_{\shortto}(n, G) \geq \mathrm{ssat}_{\shortto}(n, G) = \Omega(n).\]
\end{lem}

\begin{proof}
    We want to prove that for any $G$ we have
    $\mathrm{sat}_{\shortto}(n, G + e) = O(1) \text{ or } \Theta(n).$ Let $k=|V(G)|.$
    To prove a linear upper bound, it suffices to take as a host graph $G$ and $n-k$ isolated vertices. Adding an edge between two isolated vertices introduces a copy of $G+e$. Then we keep adding edges in the rest of the graph without introducing a copy of $G + e$ until we can. At the end, in the worst case, the number of edges in the host graph is equal to $\binom{n}{2} - \binom{n-k}{2} \leq kn$.
    
    If we assume that the saturation function of $G+e$ is sublinear, then for a large enough $n_0$, there is a host graph on $n_0$ vertices with at least two isolated vertices. These two isolated vertices can be replaced by any number of isolated vertices while preserving the saturation property. Therefore, if the saturation function is not linear, it must be bounded.

    Suppose now that $G$ has no isolated edges. Then, adding an isolated edge to the host graph $H$, no new copy of $G$ is created. This shows that the semisaturation function of $G$ is at least $\floor{n/2}$.
\end{proof}

\subsection{Related work}

The graph saturation problem was initiated by Erd\H os, Hajnal, and Moon~\cite{ErdosHajnalMoon64}. Since then, it has branched in many different directions. The survey~\cite{SatSurvey} summarizes progress and diverse aspects of saturation theory. 

In this paper, we are interested in the order of magnitude of saturation functions. It appears that these functions follow a certain dichotomy for several variants of the graph saturation problem.

\begin{thm} [Dichotomy] 
    For every unordered graph~\cite{KaszonyiTuza86}, ordered bipartite graph~\cite{FulekKeszegh21}, and vertex-ordered graph~\cite{boskesz023}, the saturation function is either bounded or linear.
\end{thm}

Ordered bipartite graphs in the context of saturation problems can be treated equivalently as $0$-$1$ matrices. They have been studied recently on several occasions \cite{Berendsohn23, BDGG25, BrualdiCao21, FulekKeszegh21, Geneson21}. Even though the full characterization of the class of matrices with bounded saturation functions remains open, for permutation matrices, they have been characterized by Berendsohn~\cite{Berendsohn23}. A characterization of vertex-ordered perfect matchings~\cite{boskesz023} (a natural analogue of permutation matrices) is still unknown.

The poset saturation problem initiated in~\cite{PosetSat21} has also attracted a lot of interest. Here, the base structure is a family of sets on $n$ vertices, and some fixed poset is forbidden in their containment poset. It is interesting to note that the analogous dichotomy has remained an open problem. Another framework in which the dichotomy is conjectured but not known is the saturation of sequences~\cite{SeqSat25}.  

In the edge-ordered setting~\cite{boskesz026}, the picture changes. The saturation function can have at least three different orders of magnitude: bounded, linear, and an undetermined function lying between $\Theta(n\sqrt{\log n})$ and $O(n\log n/{\log \log n})$. 

\smallskip
The order of magnitude of saturation functions for oriented graphs has not been studied systematically so far. However, several related variants of this problem have appeared in the literature. An early result in this direction is due to Tuza~\cite{Tuza88}, who showed that the saturation functions of digraphs, in which both orientations of an edge are allowed, are at most linear. 

A different variant was studied by Pikhurko~\cite{Pikhurko99}, where a forbidden oriented graph is acyclic, but adding an edge to the host graph either creates a copy of the forbidden graph or produces an oriented cycle. In this setting, he was able to show that the saturation function is also at most linear.

F\"uredi et al.~\cite{Furedi98} showed in a different context that an oriented triangle has a $\Theta(n \log n)$ saturation function. On the other hand, Jacobsen and Tennenhouse~\cite{JacTenn12} showed that for any oriented graph $G$, the saturation function is strictly smaller than $\binom{n}{2}$ for large enough $n$.

\subsection{Main results}

We start with some definitions.
In an oriented graph $G$, the \emph{neighborhood} $N_G(v)$ of a vertex $v$ is the set of vertices of $G$ that are connected to $v$ by an edge, regardless of its orientation.

\begin{defi}
    We say that an edge $\orarrow{uv} \in E$ has a \emph{compatible neighborhood} if for every vertex $w \in N_G(u) \cap N_G(v)$, either $\{\orarrow{wv}, \orarrow{wu}\} \subset E$ or $\{\orarrow{vw},\orarrow{uw}\} \subset E$. See Figure \ref{fig:compatible}.
\end{defi}

\begin{figure}[ht]
	\centering
	\includegraphics[width=0.5\textwidth]{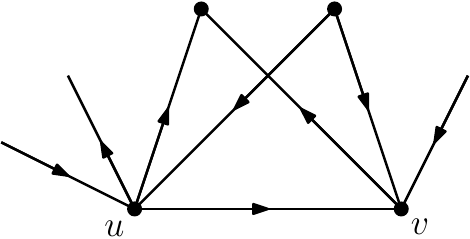}    
	\caption{\label{comp_nbhd} An edge with a compatible neighborhood.}    
    \label{fig:compatible}
\end{figure}

Our first result gives a complete characterization of oriented graphs according to their semisaturation functions.

\begin{thm} \label{thm:ssat}
    For a given oriented graph $G$, we have 
    \[  \mathrm{ssat}_{\shortto}(n,G) = 
    \begin{cases}
    O(1) & \text{if $G$ has an isolated edge,} \\
    \Theta(n) &
    \begin{aligned}
        &\text{if $G$ has no isolated edge and has an edge} \\
        &\text{with a compatible neighborhood,}
    \end{aligned} \\
    \Theta(n \log n) & \text{otherwise.}
    \end{cases}
    \]
\end{thm}

Next, we show that the saturation of acyclic oriented graphs can be characterized by the respective families. In particular, they behave in the same way as in the case of unoriented graphs:

\begin{thm} \label{thm:acyclic}
    For every acyclic oriented graph $G$, we have 
 \[  \mathrm{sat}_{\shortto}(n,G) = 
    \begin{cases}
         O(1) & \text{if $G$ has an isolated edge,} \\
         \Theta(n) & \text{otherwise.} \\         
    \end{cases}
    \]
\end{thm}

Our most interesting contribution is that the landscape is substantially richer when the graph is not acyclic. More precisely, we show the following statement.

\begin{thm} \label{thm:orders}
    There are at least four distinct orders of magnitude for the saturation functions of oriented graphs, namely
    \[O(1),\Theta(n), \Theta(n\sqrt{\log n}), \text{ and }  \Theta(n\log n).\]
\end{thm}

The appearance of the function $\Theta(n\sqrt{\log n})$ is novel for several reasons. First, it does not appear when considering semisaturation functions, see Theorem \ref{thm:ssat}. This contrasts with the unordered and vertex-ordered settings, where saturation and semisaturation functions always have the same set of orders of magnitude. For edge-ordered graphs, the two notions may differ, although whether this actually happens remains an open problem~\cite{boskesz026}. Also, while it was shown for edge-ordered graphs that at least three orders of magnitude exist, the non-standard one (the one besides $O(1)$ and $\Theta(n)$) is not completely determined. For oriented graphs, by Theorem \ref{thm:orders} we have two non-standard orders of magnitude determined exactly. The first is $\Theta(n\log n)$, found already by F\"uredi et al.~\cite{Furedi98}, while the second is $\Theta(n\sqrt{\log n})$, whose proof introduces several interesting new ideas.

As a consequence of Theorem \ref{thm:ssat} and Observation \ref{ssat < sat}, the lower saturation bound $\Omega(n \log n)$ applies to a large family of oriented graphs, which includes graphs in which every edge is contained in an oriented triangle. Proving the matching upper bound appears to be more challenging, and we identify certain classes for which it holds.

\section{Trichotomy of semisaturation functions}

In this section, our goal is to prove Theorem \ref{thm:ssat}. We start by characterizing the oriented graphs with a bounded semisaturation function.

\begin{lem} \label{ssat_bounded}
    Let $G$ be an oriented graph that contains an isolated edge. Then \[\mathrm{ssat}_{\shortto}(n,G) = O(1).\]
\end{lem}

\begin{proof}
    Let $\orarrow{uv}$ be an isolated edge in $G$.
    Let $H$ be the oriented graph on $n$ vertices, which consists of the disjoint union of three copies of the graph $G \setminus \{u, v\}$, with the remaining vertices being isolated. If we take any two non-connected vertices in $H$ and connect them with an oriented edge, then one of the three copies of $G \setminus \{u, v\}$ contains none of the endvertices of this new edge. Thus, a new copy of $G$ is created. Since $H$ has a bounded number of edges, we conclude that $\mathrm{ssat}_{\shortto}(n,G) = O(1)$.
\end{proof}

Using the notion of a compatible neighborhood, we characterize all oriented graphs with linear saturation functions.

\begin{lem} \label{ssat_n}
    Let $G$ be an oriented graph that contains no isolated edges, and there exists an edge $\orarrow{uv} \in E$ with a compatible neighborhood. Then
    \[\mathrm{ssat}_{\shortto}(n,G) = \Theta(n).\]
\end{lem}

\begin{proof}
    By Lemma \ref{dichotomy_isolated_edge}, it suffices to prove the linear upper bound.

    Let $N_G(u) \setminus \{v\} =\{u_1, u_2, ..., u_s\}$ and $N_G(v) \setminus \{u\} = \{v_1, v_2, \dots, v_t\}$, and assume $r = |V(G)|$. Let $H'$ be the graph obtained by taking the union of $G \setminus \{u, v\}$ and $n-r+2$ isolated vertices, which we denote by $\{w_1, w_2, ..., w_{n-r+2}\}$. Now, we construct the host graph $H$ by adding an edge $\orarrow{u_iw_j}$ for all $j \in \{1, \dots, n-r+2\}$, if $\orarrow{u_iu} \in E(G)$. Otherwise, we add the edges $\orarrow{w_ju_i}$ if $\orarrow{u u_i} \in E(G)$. We repeat the same procedure for all $i \in \{1, \dots, s\}$. Then, we repeat it for the vertices $\{v_1, \dots, v_t\}$, adding an edge $\orarrow{v_iw_j}$ whenever $\orarrow{v_iv} \in E(G)$ or an edge $\orarrow{w_jv_i}$ if $\orarrow{vv_i} \in E(G)$. Note that we can have $\{u_1, u_2, ..., u_s\} \cap \{v_1, v_2, ..., v_t\} \neq \varnothing$, but by the assumption, the orientation of each new edge is uniquely determined.
    
    Since adding any non-edge in $\{w_1, w_2, ..., w_{n-r+2}\}$ as an edge can have the role of the edge $\orarrow{uv}$ in a new copy of $G$, we conclude that $H$ semisaturates $G$.
\end{proof}

For our next result, we need a lower bound on the capacity of a bipartite covering of almost complete graphs. A \emph{bipartite covering} of a graph $G$ is a collection of bipartite graphs such that each edge of $G$ belongs to at least one of them. The \emph{capacity} of the covering is the sum of the numbers of vertices of these bipartite graphs. We are in fact specifically interested in the minimum number of bipartite graphs of the covering that meet at some vertex. The following statement follows easily from a result of Katona and Szemerédi~\cite{katona} (for more details, check~\cite{boskesz026} and for closely related results, see~\cite{alon,dong007}):

\begin{claim} \cite{boskesz026} \label{claim:bipdense}
	Let $G$ be a graph with at most $cn^{2-\epsilon}$ missing edges, where $c,\epsilon > 0$ are constants, and let $\cal B$ be a bipartite covering of $G$. Then there exists a vertex that is incident to edges belonging to at least $c' \log n$ bipartite graphs in $\cal B$, where $c'$ depends only on $c$ and $\epsilon$.		
\end{claim}

\begin{lem} \label{ssat_nlogn}
    Let $G$ be an oriented graph such that for each edge $\orarrow{uv} \in E$ there exists $w \in N_G(u) \cap N_G(v)$, such that either $\{\orarrow{uw}, \orarrow{wv}\} \subset E$ or $\{\orarrow{vw},\orarrow{wu}\} \subset E$. Then
    \[\mathrm{ssat}_{\shortto}(n,G) = \Theta(n \log n).\]
\end{lem}

\begin{proof}
    We first proceed with the upper bound. Take $\orarrow{uv} \in E(G)$ and let $N_G(u) \setminus \{v\} =\{u_1, u_2, ..., u_s\}$ and $N_G(v) \setminus \{u\} = \{v_1, v_2, ..., v_t\}$. We construct a host graph $H$ as a disjoint union of $k$ copies of the graph $G^{uv} = G \setminus \{u,v\}$ and $n'$ isolated vertices denoted by $W = \{w_1, w_2, \dots, w_{n'}\}$, where $k \ge 2 \ceil{\log n'}$. We can choose $k$ such that $k=O(\log n)$.

    We can now find $k$ different bipartitions of the set $W$ such that any pair of vertices is separated in at least two of them. For $i \in \{1, \dots, k\}$, we can take one bipartition $(U_i, V_i)$ of the set $W$ and one of the copies of $G^{uv}$ which we denote by $G_{i}^{uv}$. If a vertex is in the set $U_i$, then we connect it with the vertices $\{u_1, \dots, u_s\}$ of the copy $G_{i}^{uv}$ using the same orientations as $u$ is connected to the vertices $\{u_1, \dots, u_s\}$ in $G$. If a vertex is in the set $V_i$, then we connect it with the vertices $\{v_1, \dots, v_t\}$ of the copy $G_{i}^{uv}$ using the same orientation as $v$ is connected to the vertices $\{v_1, \dots, v_t\}$ in $G$. For the same $(U_i, V_i)$, we take another copy of $G^{uv}$ and do the same by switching the role of $U_i$ and $V_i$.
    
    If we add a new edge to $H$, such that one vertex is in $U_i$ and the other is in $V_i$, oriented towards the vertex in $V_i$, then together with the copy $G_{i}^{uv}$, we obtain a new copy of $G$. If it is oriented towards the vertex in $U_i$, then together with the other copy of  $G^{uv}$, we obtain a new copy of $G$.
    Since any two vertices in $W$ are separated in at least one bipartition, it follows that for any new edge, the host graph $H$ semisaturates $G$. This shows that
    \[\mathrm{ssat}_{\shortto}(n, G) = O(n \log n).\]
    
    \smallskip
    
    We continue with the lower bound. Let $H$ be a saturation host graph on $n$ vertices. For a pair of vertices $f=\{u,v\}$ that do not span an edge of $H$, adding an edge at $f$ in any orientation creates a copy of $G$. Recall that for every edge in $G$ there exists the required $w$. Thus, independently of which edge of $G$ corresponds to $f$, there is a vertex $\pi(f) \in V(H)$ such that either 
    \[\{\orarrow{u\pi(f)}, \orarrow{\pi(f)v}\} \subset E(H) \text{ or } \{\orarrow{v\pi(f)},\orarrow{\pi(f)u}\} \subset E(H).\] 
    If multiple such vertices exist, we set $\pi(f)$ to be one of them arbitrarily. For a given vertex $t\in V(H)$, consider the set of all non-edges $f$ with $\pi(f)=t$.
    Observe that $f$ must connect a vertex in $N_G^{in}(t)$ to a vertex in $N_G^{out}(t)$, the neighbors of $t$ that have an edge towards $t$ and away from $t$, respectively. Thus, there is a bipartition of the neighbors of $t$ in $H$ such that every non-edge in $\pi^{-1}(t)$ is between the two parts of this bipartition.
		
	Assume on the contrary that $H$ has less than $cn{\log n}$ edges for some $c\le 1$ chosen later. The non-edges of $H$ must be covered by the sets of non-edges $\pi^{-1}(t)$ where $t$ is a vertex of $H$. As each $\pi^{-1}(t)$ is bipartite, we can apply Claim \ref{claim:bipdense} on the complement of $H$. First, we choose constants $c_0,\epsilon_0$ independently of $c$ such that \[ n\log n\le c_0(n/2)^{2-\epsilon_0}.\] Now we can apply Claim \ref{claim:bipdense} with $c_0,\epsilon_0$.		
	We get that there is a vertex $v$ in $H$ which participates in at least $c'\log n$ many bipartitions that cover the non-edges of $H$. For each $t$ such that $\pi^{-1}(t)$ contains a non-edge $f$ with $\pi(f)=t$ and whose one endpoint is $v$, there is an edge in $H$ from $t$ to both vertices of $f$, and thus to $v$ as well. Thus $v$ is connected to at least $c'\log n$ different such vertices, therefore, to at least $c'\log n$ vertices in $H$.
    
    Delete $v$ from $H$ and repeat this procedure $n/2$ times. In each step we can apply Claim \ref{claim:bipdense} with the same $c_0$ and $\epsilon_0$ to find and delete a vertex with degree at least $c'\log (n/2)$. Altogether, we have found at least $\frac{n}{2} c'\log (n/2)$ edges in $H$. When $c$ is small enough this contradicts our assumption that $H$ has less than $cn\log n$ edges (note that $c'$ was independent from $c$).    
\end{proof}

We note that the lower bound works the same way in a more general, albeit more technical setting, which we state as we will need it later:

\begin{cor} \label{cor:subsat}
    Let $G$ be an oriented graph such that for each edge $\orarrow{uv} \in E$ there exists $w \in N_G(u) \cap N_G(v)$, such that either $\{\orarrow{uw}, \orarrow{wv}\} \subset E$ or $\{\orarrow{vw},\orarrow{wu}\} \subset E$. Let $H$ be a graph with vertex set $V=V_1\cup V_2$ such that for every non-edge of $V_2$ adding this edge would create a new copy of $G$ in $H$. Then $H$ has $\Omega(|V_2|\log |V_2|)$ edges.
\end{cor}

Combining the previous three lemmas implies Theorem \ref{thm:ssat}. To give a sense of the assumption of compatible neighborhoods, we mention that Lemma \ref{ssat_nlogn} implies the following bound.

\begin{cor}\label{cor:triangles}
Let $G$ be an oriented graph such that every edge is in some oriented triangle. Then \[\mathrm{ssat}_{\shortto}(n,G) = \Theta(n \log n).\]
\end{cor}

\section{Saturation functions}

In this section, we proceed to systematically study the saturation functions. We distinguish two natural families of oriented graphs, depending on whether they contain an isolated edge. 

\subsection{Oriented graphs with an isolated edge} \label{with_isolated_edge}

Recall that by Lemma \ref{dichotomy_isolated_edge}, when there is an isolated edge in the graph, the saturation function is either bounded or linear. Graphs that are tournaments (i.e., a complete oriented graphs) together with an isolated edge will play a distinguished role.

\begin{lem} \label{k4_bounded}
    Let $G$ be an oriented graph on $k$ vertices such that $\mathrm{sat}_{\shortto}(n,G + e) = O(1)$. Then, its host graph with $n$ vertices and $\mathrm{sat}_{\shortto}(n,G + e)$ edges has at most $k^2$ non-isolated vertices. Consequently, $\mathrm{sat}_{\shortto}(n,G + e) = O(k^4)$.
\end{lem}
\begin{proof}
    Let $H$ be a host graph with a bounded number of edges and fix an isolated vertex $v$. If we connect it to another isolated vertex, we obtain in the rest of $H$ a copy of $G$, which we denote by $G'$. Since $H$ is $(G+e)$-free, $V(H) \setminus V(G')$ spans an independent set. Therefore, every edge in $H$ contains at least one vertex in $G'$.

    Let $u \in V(G')$ and add the edge $\orarrow{uv}$ to $H$. We then obtain a copy of $G+e$, which has at most $k-1$ non-isolated vertices outside of $G'$. Indeed, if $\orarrow{uv}$ plays the role of the isolated edge $e$, then at least one vertex of the copy of $G$ must be in $G'$, so at most $k-1$ of its vertices lie outside $V(G')$. On the other hand, if $\orarrow{uv}$ plays the role of an edge from $G$, then there are at most $k-2$ vertices of $G$ and one vertex of $e$ outside of $G'$, as desired.
    
    By repeating this argument for every vertex in $G'$, we get at most $k(k-1)$ vertices outside $V(G')$. Altogether, there are at most $k^2$ non-isolated vertices, including $V(G')$.
\end{proof}

In fact, if we consider the graphs $T+e$ where $T$ is a tournament, we can have a more precise description of the host graphs.

\begin{lem} \label{r+1 noniso}
    Let $T$ be a tournament on $r\ge 2$ vertices. If $\mathrm{sat}_{\shortto}(n,T + e) = O(1)$, then for large enough $n$, its host graphs with $n$ vertices and $\mathrm{sat}_{\shortto}(n,T + e)$ edges have exactly $r+1$ non-isolated vertices that span a tournament, and every subset of $r$ vertices of this tournament spans a copy of $T$. Consequently, $\mathrm{sat}_{\shortto}(n,T + e) = \binom{r+1}{2}$.
\end{lem}
\begin{proof}
    Let $H$ be a host graph with a bounded number of edges. By assumption, there exists a subgraph $T' \subset H$ that is isomorphic to $T$. If the host graph has at most $r$ non-isolated vertices, then it cannot saturate edges between an isolated vertex and a non-isolated vertex.
    Note that the vertices $V(H)\setminus V(T')$ must span an independent set, therefore any copy of $T$ can have at most one vertex from this set.

    Fix an isolated vertex $w$. For every $u \in V(T')$, adding the edge $\orarrow{uw}$ creates a copy of $T + e$ such that $T$ is isomorphic to the graph spanned over the vertices of $T'-u$ and one of the vertices $v_u \notin V(T')$. 

    If for every $u\in V(T')$ these $v_u$ vertices are all the same, then we get a tournament on $r+1$ vertices such that every subset of $r$ vertices spans a copy of $T$. Notice that there are no further edges in the graph, as we would then get a copy of $T+e$.

    Otherwise, assume there exist $u, u' \in V(T')$ such that $v_u\ne v_{u'}$. In this case, $v_u$ is connected to every vertex of $T'-u$, thus also to $u'$. The edge between $u'$ and $v_u$ together with the copy of $T$ on $V(T'-u')\cup \{v_{u'}\}$ forms a copy of $T+e$, which is a contradiction.
\end{proof}

\begin{cor} \label{subgraph of a tournament}
    Let $T$ be a tournament on $r$ vertices, and let $S$ be a spanning subgraph of $T$. If $\mathrm{sat}_{\shortto}(n,T + e) = O(1)$, then $\mathrm{sat}_{\shortto}(n,S + e) = O(1)$.
\end{cor}

\begin{proof}
    By Lemma \ref{r+1 noniso} large enough host graphs $H$ of $T+e$ have at most $r+1$ non-isolated vertices. Therefore, they also do not contain $S$, which has $r+2$ non-isolated vertices. On the other hand, adding any edge to $H$ introduces a copy of $T+e$, and thus also of $S+e$.
\end{proof}

This corollary gives us a way to find many examples with a bounded saturation function. However, if the saturation function of $T+e$ is linear, it remains unclear how the saturation function of $S+e$ behaves.

\begin{cor} [Algorithmic characterization]
    We observe that Lemmas \ref{k4_bounded} and \ref{r+1 noniso}, together with the dichotomy obtained in Lemma \ref{dichotomy_isolated_edge}, provide an algorithmic characterization of the saturation functions of oriented graphs of the form $G+e$. Namely, if $G$ has $k$ vertices, we can check all oriented graphs on $n$ vertices with at most $k^2$ non-isolated vertices to determine whether any of them saturates $G+e$. If such a graph exists, then the algorithm also determines the exact saturation function. Otherwise, we conclude that the saturation function is linear. Note that this algorithm is exponential in $k$. 
    
    For a tournament $T$, this algorithm can be simplified. It suffices to take a copy of $T$ and add a single vertex, considering every possible orientation of the edges between this vertex and the vertices of $T$.
\end{cor}

As stated by Lemma \ref{dichotomy_isolated_edge}, we have a dichotomy for saturation functions of oriented graphs containing an isolated edge. Using Lemma \ref{r+1 noniso}, we can find a simple algorithm to determine the order of magnitude of saturation functions for many small tournaments. For a tournament on $r$ vertices, it suffices to take a tournament on $r+1$ vertices, orient its edges in every possible way, and check whether each subtournament on $r$ vertices is isomorphic to the initial tournament. However, we do not make such an attempt here. Instead, we identify some general families that have bounded or linear saturation functions.

\begin{claim} \label{claim:e+acyclic}
    Let $G$ be an acyclic oriented graph. Then \[\mathrm{sat}_{\shortto}(n,G + e) = O(1).\]
\end{claim}

\begin{proof}
    Assume $G$ has $k \geq 2$ vertices. Then we take a host graph $H$ to be a transitive tournament on $k+1$ vertices, and the remaining $n-k-1$ vertices are isolated. Since $|V(H)| < |V(G+e)|$, $H$ avoids $G+e$. Let $u, v \in V(H)$ be such that $uv \notin E(H)$. Then we can assume that one of them is isolated, say $u$. Now, we can observe that $H - v$ contains a transitive tournament of size $k$. Therefore, since $G$ is a spanning subgraph of a transitive tournament, $H - v$ also contains $G$. By taking $e = uv$, we conclude that $H + uv$ contains $G + e$.
\end{proof}

\begin{claim}
    Let $T$ be a tournament with outdegree sequence $d_1 \leq d_2 \leq \dots \leq d_r$. If there exists $k \in \{1, \dots, r-1\}$ such that $d_k + 2 \leq d_{k+1}$, then \[\mathrm{sat}_{\shortto}(n,T + e) = \Theta(n).\]
\end{claim}

\begin{proof}
    By Lemma \ref{r+1 noniso}, it suffices to show that there is no tournament $T_{r+1}$ on $r+1$ vertices such that any subset of $r$ vertices induces a copy of $T$. Suppose such a tournament $T_{r+1}$ exists. Choose one vertex $u$, and label the remaining $r$ vertices $v_1, \dots, v_r$ so that they induce a copy of $T$ in which $v_i$ has outdegree $d_i$.

    Moreover, $T_{r+1}-v_k$ must induce a copy of $T$. We get this from the earlier copy of $T$ by removing $v_k$ and adding $u$. Note that no $v_i$ with $i > k$ can play the role of a vertex $v_j$ with $j \leq k$, since deleting a single vertex decreases the outdegree of every remaining vertex by at most one. Therefore, after removing $v_k$, the vertex $u$ must play the role of a vertex with outdegree at most $d_k$. Consequently, $d^+(u)\le d_k+1$.
    
    Similarly, in the copy of $T$ induced by $T_{r+1}-v_{k+1}$, the vertex $u$ must correspond to a vertex with outdegree at least $d_{k+1}$. In particular, $d^+(u)\geq d_{k+1}$.
    We get that \[d_{k+1}\le d^+(u)\leq d_k+1<d_{k+1},\] 
    a contradiction. 
\end{proof}

\begin{defi}
    An oriented graph $G$ is \emph{deletion-homogeneous} if, for any pair of vertices $u, v$, the graphs $G - u$ and $G - v$ are isomorphic.
\end{defi}

Any vertex-transitive oriented graph is deletion-homogeneous. Note that transitive tournaments are deletion-homogeneous but not vertex-transitive. They are not the only examples with this property. Figure \ref{6 vertices del hom} shows a tournament on six vertices that is deletion-homogeneous yet not vertex-transitive. It would be interesting, as an independent question, to explore these graphs further. This is also motivated by the next statement.

\begin{figure}[ht]
	\centering
	\includegraphics[width=0.25\textwidth]{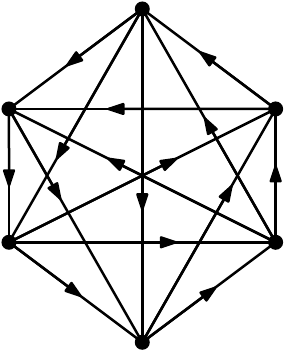}
	\caption{\label{6 vertices del hom} A deletion-homogeneous tournament which is not vertex-transitive.}
\end{figure}

\begin{claim}
    Let $r \geq 1$ and let $G$ be a deletion-homogeneous graph on $r+1$ vertices. If $G'$ is a subgraph of $G$ on $r$ non-isolated vertices, then $\mathrm{sat}_{\shortto}(n,G' + e) = O(1)$.
\end{claim}

\begin{proof}
    We can take a host graph $H$ of $G' + e$ to be $G$ and $n-r-1$ isolated vertices. Since $G' + e$ has $r+2$ non-isolated vertices, it is certainly avoided by this construction. Furthermore, by assumption, $G$ is deletion homogeneous, so $G - v$ is isomorphic to $G'$ for any vertex $v$. We conclude that by adding any edge $uv$, where $u$ is an isolated vertex of $H$, we introduce a copy of $G' + e$.
\end{proof}

\begin{defi}
    Let $G_1$ and $G_2$ be two oriented graphs. We denote by $G_1 \circ G_2$ the oriented graph obtained by replacing every vertex of $G_1$ with a copy of $G_2$, and if $\overrightarrow{uv} \in E(G_1)$ then $\overrightarrow{u'v'} \in E(G_1 \circ G_2)$ for all $u' \in G_2^u$ and $v' \in G_2^v$, where $G_2^u$ and $G_2^v$ are the copies of $G_2$ that replace $u$ and $v$.
\end{defi}

Building on the above mentioned deletion-homogeneous graphs, we present a general way to construct further examples.

\begin{prop}
    Let $G_1$ be a vertex-transitive graph and $G_2$ be a deletion-homo\-geneous graph. Then $G_1 \circ G_2$ is a deletion-homogeneous graph.
\end{prop}

\begin{proof}
    Let $v_1, v_2 \in V(G_1 \circ G_2)$, and let $G'_2$ and $G''_2$ be two disjoint copies of $G_2$ in $G_1 \circ G_2$. We can distinguish two cases.
    
    First, suppose $v_1, v_2 \in V(G'_2)$. Then we can see that $G_1 \circ G_2 - v_1$ and $G_1 \circ G_2 - v_2$ are isomorphic because the graph $G_2$ is deletion-homogeneous and the isomorphism is given by swapping the roles of $v_1$ and $v_2$, while keeping the  other vertices the same. 
    
    Second, suppose $v_1 \in V(G'_2)$ and $v_2 \in V(G''_2)$. Since $G_2$ is deletion-homogeneous, we can find an isomorphism between $G'_2 - v_1$ and $G''_2 - v_2$. By the assumption that $G_1$ is vertex-transitive, we conclude that $G_1 \circ G_2 - v_1$ and $G_1 \circ G_2 - v_1$ are isomorphic.
\end{proof}

\subsection{Oriented graphs without isolated edges}

Recall that by Lemma \ref{dichotomy_isolated_edge}, when there are no isolated edges in the graph, the saturation function is at least linear. First, we focus on the case of acyclic oriented graphs and observe that we can adapt the proof of Pikhurko~\cite{Pikhurko99} in a straightforward way. We include the proof for the sake of completeness.

\begin{lem} \label{acyc_linear}
    For every acyclic oriented graph $G$ without isolated edges, we have 
    \[\mathrm{sat}_{\shortto}(n,G) = \Theta(n).\]
\end{lem}
\begin{proof}
    By Lemma \ref{dichotomy_isolated_edge}, we only need an upper bound construction.
    We define a host graph $H$ on the vertex set $\{v_1, \dots, v_n\}$. Starting with the empty graph, we add edges one-by-one in the following way. We take the first vertex $v_1$ and keep adding edges $\orarrow{v_1v_j}$ for $1 < j \leq n$ either until a copy of $G$ is created (in which case we do not add this last edge) or $j=n$. Then we do the same for further vertices $v_2,v_3,\dots$ one-by-one. In a general step, we add edges from some $v_i$ to $v_j$ with $i<j\leq n$ until a copy of $G$ is created (in which case we do not add this last edge) or $j=n$.
    
    This procedure stops (that is, adds no more edges) at some vertex $v_x$ after adding $y\ge 0$ edges $v_xv_j$ with $x<j< n$. First, $x\le k-1$, otherwise when we add the edge $\orarrow{v_{k-1}v_k}$, a copy of the transitive tournament on $k$ vertices is created, which contains $G$. Next, $y< k$, otherwise in the last step, when adding the edge $\orarrow{v_{x}v_{x+y+1}}$, we would create a copy of $G$, necessarily using the new vertex $v_{x+y+1}$. However, the vertices $v_{x+1},\dots v_{x+y}$ all have the same neighborhoods as $v_{x+y+1}$ would have. Thus, we could replace $v_{x+y+1}$ in this copy of $G$ by one of $v_{x+1},\dots,v_{x+y}$ that does not belong to the copy (which is possible since $y\ge k$), showing that a copy of $G$ already existed before the last step.

    We obtain a graph $H_0$ with at most $(k-1)n$ edges such that adding any further edge incident to $v_x$ oriented away from $v_x$ creates a graph $H_0'$ containing $G$. 
    We now greedily extend $H_0$ until it saturates $G$. Let $H$ denote this graph.
    Note that in $H$, no vertex $v_i$, $i>x$ can have an out-degree larger than $y$, as otherwise this vertex could play the role of $v_x$ in a copy of $H_0'$ (while all other vertices are the same) and therefore would also contain $G$. Thus, as $y< k$, we added at most $(k-1)n$ further edges.

    Therefore, the saturation function of $G$ is at most $2(k-1)n$, which completes the proof.
\end{proof}

We note that the original proof of the dichotomy for undirected graphs \cite{KaszonyiTuza86} can also be adapted to our case, but it is more technical. Also, we did not aim to give the best upper bound, whereas they had a bound dependent on an appropriate parameter of the graph.

\begin{proof}[Proof of Theorem \ref{thm:acyclic}]
    It suffices to use Claim \ref{claim:e+acyclic} and Lemma \ref{acyc_linear} to obtain a full characterization of saturation functions for acyclic oriented graphs.
\end{proof}

\begin{lem}
    Let $G$ be an oriented graph that contains a unique vertex $x \in V(G)$ with indegree (or outdegree) zero. Then
    \[\mathrm{sat}_{\shortto}(n,G) \leq \mathrm{sat}_{\shortto}(n,G - x) + O(n).\]
\end{lem}

\begin{proof}
    We take a host graph $H_{G-x}$ of $G-x$ on $n-1$ vertices. To obtain a host graph $H$ of $G$, we add a vertex $v$ to $H_{G-x}$ and add an edge $v \rightarrow u$ for all $u \in V(H_{G-x})$. 
    First, assume that $H(G)$ contains $G$. As $H_{G-x}$ avoids $G$, the only way for $H(G)$ to contain a copy of $G$ is to have $v$ in it. Since $v$ has a zero indegree, it can only have the role of the vertex $x$. That would mean $H_{G-x}$ contains a copy of $G-x$, which is not possible. Second, $H(G)$ saturates $G$ because $H_{G-x}$ saturates $G-x$.
\end{proof}

We denote by $G_1 \rightarrow G_2$ the family of graphs obtained from two connected oriented graphs $G_1$ and $G_2$ by taking their disjoint union and then choosing vertices $v_1\in V(G_1), v_2\in V(G_2)$ arbitrarily and adding the edge $\orarrow{v_1v_2}$.

\begin{lem} \label{ver_trans}
    Let $G$ be a vertex-transitive graph, and let $G_1$ and $G_2$ be connected subgraphs of $G$. If $|V(G_1)|+|V(G_2)| \geq |V(G)|+1$, then 
    \[\mathrm{sat}_{\shortto}(n, F) = \Theta(n),\]
    for every $F \in G_1 \rightarrow G_2$.
\end{lem}

\begin{proof}
    As $F$ is a connected graph, its saturation function is at least linear. On the other hand, we can construct a host graph by taking linearly many disjoint copies of $G$. Note that $F$ is connected and has more vertices than $G$. Thus, we obtain that each copy of $G$ avoids $G_1 \rightarrow G_2$, and therefore the host graph does so too.
    
    It remains to show that this graph is saturated. If we add an edge between two different copies of $G$, then, by vertex-transitivity, we can find a copy of $G_1$ in one of them and a copy of $G_2$ in the other. Together with the added edge, they form a copy of $F$.
\end{proof}

\begin{cor}
    Let $G$ be a vertex-transitive graph. Let $G_1$ be a spanning subgraph of $G$, and $G_2$ a connected subgraph of $G$. Then
    \[\mathrm{sat}_{\shortto}(n, F) = \Theta(n),\]
    for every $F \in G_1 \rightarrow G_2$.
\end{cor}

\begin{proof}
    It suffices to observe that $|V(G)| = |V(G_1)|$, and so the inequality from Lemma \ref{ver_trans} is satisfied.
\end{proof}

As direct consequences of Lemma \ref{ssat_n} and Lemma \ref{ssat_nlogn} about semisaturation, we obtain the same lower bounds for the saturation functions of the respective families of oriented graphs. It is a natural question to investigate whether we can also match the upper bounds.

\begin{lem} \label{sat_nlogn}
    Let $G$ be an oriented graph with edge set $E$ such that for each edge $\orarrow{uv} \in E$ there exists $w \in N_G(u) \cap N_G(v)$, such that either $\{\orarrow{uw}, \orarrow{wv}\} \subset E$ or $\{\orarrow{vw},\orarrow{wu}\} \subset E$. If in addition there exists $\orarrow{xy} \in E$ such that both $N_G(x) \setminus \{y\}$ and $N_G(y) \setminus \{x\}$ form independent sets, then
    \[\mathrm{sat}_{\shortto}(n,G) = \Theta(n \log n).\]
\end{lem}

\begin{proof}
    The lower bound follows from Lemma \ref{ssat_nlogn}.
    For the upper bound, we use the same host construction $H$ as the one provided in the proof of Lemma \ref{ssat_nlogn} using $\orarrow{xy}$ as the distinguished edge. However, we apply the construction to only one of the largest connected components of $G$. It remains to show that this construction indeed avoids the oriented graph $G$.
    
    We assume the contrary and notice that any copy should contain at least one edge incident to the vertex set $W$ (the independent set defined in the construction). Indeed, in $H-W$ there are not enough components of $G$ of the largest size to contain a copy of $G$. So, assume there is a vertex $w \in W$ in the copy of $G$. Notice that by the assumption, $N_H(w)$ forms an independent set since it inherits the neighbors of $x$ and $y$ (with some vertices taken in each copy of the largest component), which both form independent sets. This implies that any edge incident to $w$ is not part of a triangle, contradicting the assumption that each edge in $G$ is in some triangle. See Figure \ref{Gandtriangle} for a special case of this construction.
\end{proof}

\begin{cor}
    Let $G$ be an oriented graph such that every edge belongs to an oriented triangle. If there exists an edge $\orarrow{uv} \in E(G)$ such that $\deg(u) = \deg(v) = 2$, then
    \[\mathrm{sat}_{\shortto}(n,G) = \Theta(n \log n).\]
\end{cor}

\begin{proof}
    It follows as a direct application of Lemma \ref{sat_nlogn}. Figure \ref{Gandtriangle} shows how a host graph with $\Theta(n \log n)$ edges can be constructed for such graphs.
\end{proof}

\begin{figure}[ht]
    \centering
	\includegraphics[width=0.65\textwidth]{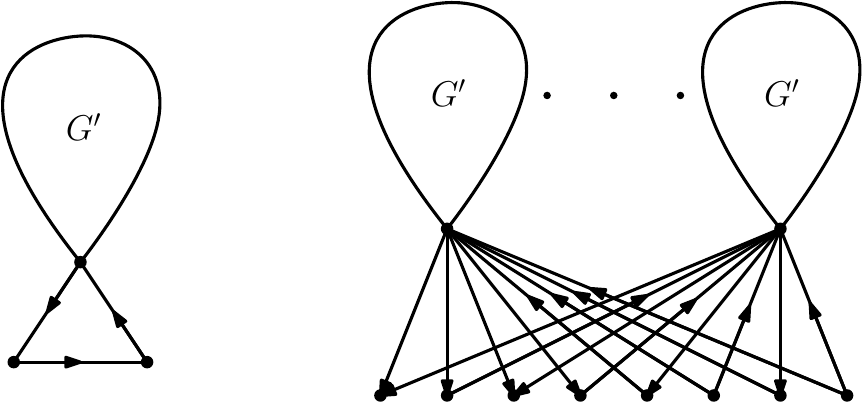}
	\caption{\label{Gandtriangle} Left: the graph $G$ that consists of the graph $G'$ and an oriented triangle meeting at one vertex. Right: the host graph obtained using the copies of $G'$ and adding properly oriented edges with the independent set $W$ drawn at the bottom.}
\end{figure}

\begin{lem} \label{bip nlogn upper}
    Let $G$ be an oriented graph such that there exists an edge $\orarrow{uv} \in E$ for which $N_G(u) \cap N_G(v) \neq \emptyset$, and the graph obtained from $G$ by removing the edge $\orarrow{uv}$ is bipartite. Then
    \[\mathrm{sat}_{\shortto}(n,G) = O(n \log n).\]
\end{lem}
\begin{proof}
    First observe that $G$ is non-bipartite, as it contains a triangle. We use the same construction of the host graph $H$ as in Lemma \ref{ssat_nlogn} with distinguished edge $\orarrow{uv}$, which semisaturates $G$. It remains to show that $H$ avoids $G$. By assumption, the graph obtained from $G$ by removing the edge $\orarrow{uv}$ is bipartite. We can partition its vertices into $B$ and $W$. Since $N_G(u) \cap N_G(v)$ is non-empty, we know that $u$ and $v$ are in the same part, wlog. $B$. Then, it follows that $N = (N_G(u) \cup N_G(v)) \setminus \{u,v\} \subseteq W$.

    Following the construction in Lemma \ref{ssat_nlogn}, we take $k$ copies of the graph $G^{uv} = G \setminus \{u,v\}$ and $n'$ isolated vertices. Then, for each copy of $G^{uv}$, we take the corresponding set $N$ and connect it to the isolated vertices. Since all the vertices of $N$ are in $W$, we can take the isolated vertices to be in $B$. This preserves the bipartiteness of the host graph $H$. As $G$ is non-bipartite, it follows that $H$ contains no copy of $G$. 
\end{proof}

For the rest of this section, we focus on specific examples of oriented graphs. Let $D_2$ be the oriented graph shown in Figure \ref{D0_D1_D2_H1}, and $W_2$ (resp. $Q_2$) obtained by adding an extra outgoing (resp. ingoing) leaf edge to $D_2$. We can generalize $D_2, W_2, Q_2$ to $D_s, W_s, Q_s$ by replacing the two parallel oriented paths of length two with $s$ such paths. Remarkably, the saturation functions of these three families have different orders of magnitude.

\begin{figure}[ht]
	\centering
	\includegraphics[width=0.65\textwidth]{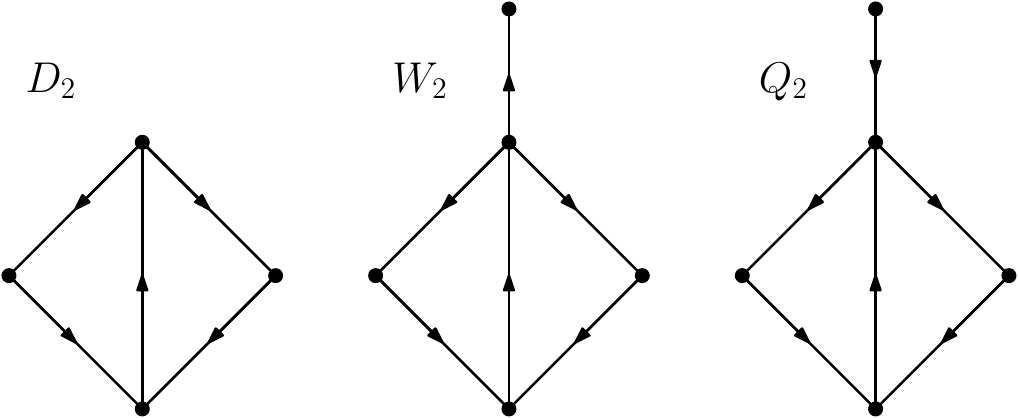}
	\caption{\label{D0_D1_D2_H1} Three oriented graphs whose saturation functions have different orders of magnitude.}
\end{figure}

We first observe that the saturation function of $D_s$ can be obtained as a direct consequence of Lemma \ref{ssat_nlogn} and Lemma \ref{bip nlogn upper}.

\begin{cor}
    For $s \ge 1$,
    \[\mathrm{sat}_{\shortto}(n,D_s)=\Theta(n \log n).\]
\end{cor}

We now show constructively the linearity of the saturation function of $W_s$.

\begin{prop}
    For $s \ge 1$,
    \[\mathrm{sat}_{\shortto}(n,W_s)=\Theta(n).\]
\end{prop}
\begin{proof}
    It suffices to construct a host graph with a linear number of edges. Let $X_{s}$ be a tournament on vertex set $V(X_s) = \{v_{1}, \dots, v_{2s+1}\}$ and edge set \[E(X_s) = \{\orarrow{v_{i}v_{j}} : 1 \leq j - i \leq s \pmod{2s+1}\}.\] 

    Let $H$ be a host graph on $n$ vertices obtained from a disjoint union of $\floor{\frac{n}{2s+1}}$ copies of $X_{s}$ and an independent set $I$ of isolated vertices.

    Notice that $H$ avoids $W_s$ because there exists a vertex in $W_s$ with outdegree $s+1$, while the maximum outdegree in $H$ is $s$. On the other hand, we want to show that adding any edge within $H \setminus I$ creates a copy of $W_s$. By symmetry, it suffices to consider an edge going out of the vertex $v_k$, for $1 \leq k \leq 2s+1$. We want to argue that the subgraph induced by the vertices $v_k, \dots, v_{k+s+1}$ contains a copy of $D_s$. By the construction of $X_{s}$, there is a $2$-path $v_k \rightarrow v_i \rightarrow v_{k+s+1}$ for every $i \in \{k+1, \dots, k+s\}$. Moreover, $\orarrow{v_{k+s+1}v_{k}} \in E(X_{s})$, which indeed forms a copy of $D_s$. We can greedily add linearly many edges from $I$, until the host graph fully saturates $W_s$.
\end{proof}

\begin{figure}[ht]
	\centering
	\includegraphics[width=0.35\textwidth]{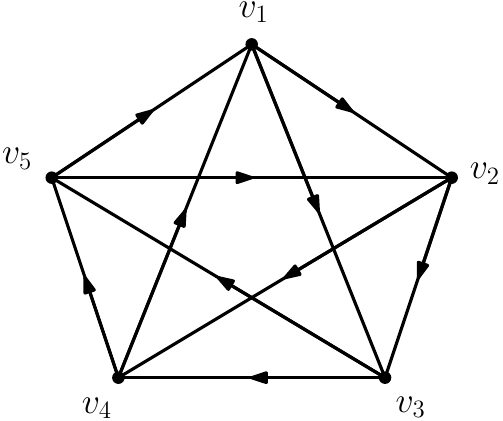}
	\caption{\label{X2} The graph $X_2$. The host graph $H$ that saturates $W_2$ is obtained as a disjoint union of copies of $X_2$ and an independent set.}
\end{figure}

Finally, we end this section by determining the order of magnitude of the saturation functions for $Q_s$, where $s \geq 2$. The proof is divided into two parts: lower and upper bounds. The lower bound relies on a subtle application of the tools introduced earlier, while the upper bound is obtained through a non-trivial construction.

\begin{thm}
    For $s\ge 2$, \[\mathrm{sat}_{\shortto}(n,Q_s)=\Theta(n\sqrt{\log n}).\footnote{The saturation function of $Q_1$ is linear, as follows from the same construction used for $W_1$.}\]
\end{thm}
\begin{proof}    
    We first prove the lower bound. Let $H$ be an $n$-vertex host graph saturating $Q_s$, with vertex set $V$. We denote by $D_s$ the graph obtained from $Q_s$ by deleting the leaf vertex and the incident edge. An \emph{instar} is the graph obtained from $D_s$ by deleting its top vertex, which is adjacent to the leaf vertex of $Q_s$. The unique vertex of indegree $s$ in an incherry is called its \emph{center}.
    
    Every non-edge plays the role of some edge in a copy of $Q_s$ once it is added to $H$. Let $V_1$ consist of those vertices that are incident to a non-edge that can serve as the hanging edge of a copy of $Q_s$ in which this vertex is not the leaf vertex of this copy of $Q_s$. Note that the vertices of $V_1$ have indegree exactly one, since a larger indegree would create a copy of $Q_s$ in $H$. 

    Let $V_2 = V \setminus V_1$. By the definition of $V_1$, for each $v\in V_1$, there exists a copy of $D_i$ in $H$ such that $v$ corresponds to its top vertex, and there is an instar on the remaining $s+1$ vertices of this copy of $D_s$. Note that while for a given $v_1\in V_1$, this instar might not be unique, its center is unique. For simplicity, we can fix one such instar for each vertex of $V_1$.

    First, we claim that every vertex belonging to one of these instars must be in $V_2$. Indeed, if an instar contained a vertex of $V_1$, using the fact that every vertex $V_1$ has indegree exactly $1$, we obtain a copy of $Q_s$. See the rightmost figure of Figure \ref{fig:incherries} for $s=2$.
    
    Second, we claim that each edge $e$ is an edge of at most $s-1$ defined above instars. Indeed, each such instar contains an oriented triangle through $e$. If there are $s$ instars through $e$, then there would be $s$ oriented triangles. Together with an additional edge from one of these instars playing the role of the leaf, we would get a copy of $Q_s$, which is a contradiction. See Figure \ref{fig:incherries} for $s=2$.

    \begin{figure}[ht]
    	\centering
    	\includegraphics[width=0.8\textwidth]{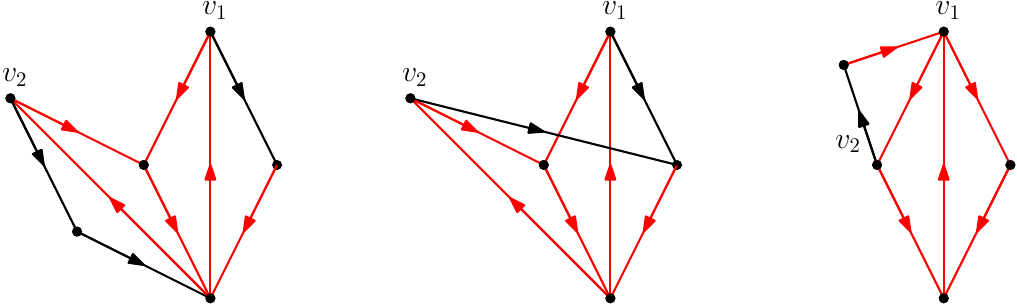}
    	\caption{\label{fig:incherries} If an edge is in at least $s=2$ instars (left and middle figure) or if one instar has a vertex $v_2$ in $V_1$ (right figure), then $H$ contains $Q_s$ (red).}
    \end{figure}

    Observe that within $V_2$, no non-edge can be saturated such that it serves as the hanging edge of $Q_s$. Hence, for every non-edge in $V_2$, adding this edge would create a copy of $D_s$. By Corollary \ref{cor:subsat}, $H$ has $\Omega(|V_2|\log |V_2|)$ edges. If $|V_2|=\Omega(n/\sqrt{\log n})$, then we are done. 
    
    From now on, we may assume that $N=|V_2|=O(n/\sqrt{\log n})$. 
    For each vertex of $V_2$, let $a_i$ denote the number of instars defined above that have this vertex as their center, and order these numbers so that
    \[a_1\le a_2\le \dots \le a_{N}\le \frac{(s-1)(N-1)}{s}.\]
    The last inequality follows because each instar consists of $s$ edges, all contained in $V_2$, and each edge belongs to at most $s-1$ instars. Since we fixed exactly one instar for each vertex of $V_1$, we have
    \[\sum_{1 \leq i \leq N} a_i = |V_1| = n-N.\] 
    Moreover, since $N=O(n/\sqrt{\log n})$, we have $n-N\ge n/2$ for all sufficiently large $n$. 
    
    We claim that for any pair of vertices $w_i, w_j\in V_2$, with incident instars corresponding to some $v_i,v_j\in V_1$, respectively, if $\orarrow{w_jw_i}\notin E(H)$, then $w_j$ and $v_i$ must be adjacent in $H$. Indeed, otherwise, the non-edge $\orarrow{v_iw_j}$ needs to be saturated. Adding this edge cannot serve as the leaf edge of a copy of $Q_s$ (as $w_j$ is in $V_2$). Thus, it must complete an oriented triangle, which would use the edge $\orarrow{w_iv_i}$, as $v_1$ has indegree one. This would require an edge $\orarrow{w_jw_i}$ in $H$, which does not exist. See Figure \ref{fig:wiwjforceedge}.

    \begin{figure}[ht]
    	\centering
    	\includegraphics[width=0.4\textwidth]{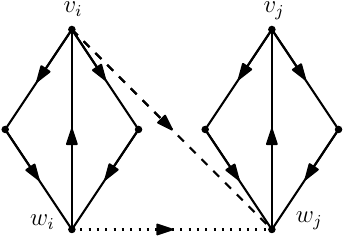}
    	\caption{\label{fig:wiwjforceedge} No edge from $w_i$ to $w_j$ forces an edge $v_iw_j$ in $H$.}
    \end{figure}

    For a fixed vertex $w_i$, there are $a_i$ corresponding vertices in $V_1$. Since there cannot be an edge oriented in both directions between $w_i$ and $w_j$, the previous argument implies that there are at least $\min \{a_i,a_j\}$ edges between $V_1$ and $\{w_i,w_j\}$. Therefore, the number of edges between $V_1$ and $V_2$ is at least $\sum_{1\le i<j\le N}\min \{a_i,a_j\}$. We are left to give a lower bound to \[A=\sum_{1\le i<j\le N}\min \{a_i,a_j\} = (N-1)a_1+(N-2)a_2+\dots 2a_{N-2}+a_{N-1}.\]
        
    Recall that for large enough $n$, we have $\sum a_i\ge  n/2$, and $a_N\le N\le cn/\sqrt{\log n}$ for some constant $c$. Assuming only these, we get that the sum $A$ is smallest when $N$ is as large as possible, i.e., $N=\lfloor cn/\sqrt{\log n}\rfloor$, and the $a_i$ are $0$ up to some point, after which they are all equal to $N$ (with one possible exception). In this case, we have at least $x= \lfloor n/(2N)\rfloor$ values of $a_i$ that are equal to $N$. At least $\lfloor x/2\rfloor-1$ of them are multiplied by a factor of at least $\lfloor x/2\rfloor$ in $A$. Consequently,
    \[A \geq \lfloor x/2\rfloor(\lfloor x/2\rfloor-1) N = \Omega(n^2/N) = \Omega (n\sqrt{\log n}).\] 
    Hence, the host graph $H$ contains at least this many edges, finishing the proof of the lower bound.
\bigskip    

    We now present the construction showing the matching upper bound. The idea is to provide a construction for which the upper bound would be tight at each step. First, we saturate all non-edges in some subset $V_1 \subseteq V$ that contains most vertices. Then, we saturate most non-edges in the remaining set, as done in Lemma \ref{ssat_nlogn}, using only a small set of vertices so that we do not need to make further iterations.

    We first construct an almost saturating graph $H_0$ on $n$ vertices, in which all but a few non-edges are saturated in both orientations. The desired saturated graph $H$ is then obtained by greedily orienting the remaining non-edges until the graph becomes saturated.

    \begin{figure}[ht]
        \centering
        \includegraphics[width=1\linewidth]{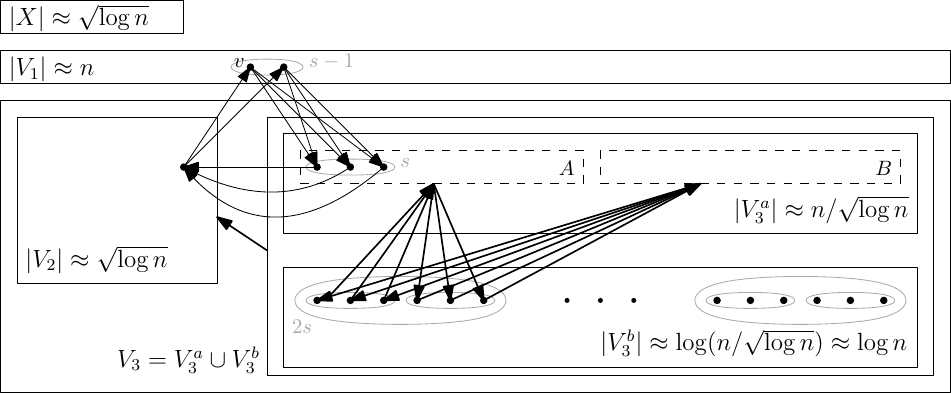}
        \caption{The host graph $H_0$ that (almost) saturates $Q_3$.}
        \label{fig:diamondplus}
    \end{figure}

    For the structure of $H_0$, see Figure \ref{fig:diamondplus}, which shows the case $s=2$, the general case being similar. We may assume that $n$ is large enough.
    Let $y=\lceil \sqrt{\log n}\rceil $ and let $V_2\subset V $ such that $|V_2|=y$.
    We partition the remaining vertices into disjoint sets $V_1$, $V_3$, and $X$ such that $|V_3|$ is divisible by $s$ and 
    \[ y\cdot |V_3|\frac{s-1}{s}=|V_1|,
    \qquad
    |X|<(s-1)y+s.\]
    This can be easily done by starting with $V_1=V\setminus V_2$ and moving vertices from $V_1$ to $V_3$ in sets of size $s$, as long as $y\cdot|V_3|(s-1)/s\le |V_1|$. We stop just before this would become false. Note that at this point for $x=|V_1|-y\cdot|V_3|(s-1)/s$, we have $0\le x < (s-1)y+s=O(\sqrt {\log n})$. Now take $x$ vertices from $V_1$ to form $X$. Finally, note that 
    \[|V_3|\le \frac{ns}{(s-1)y} \le \frac{2n}{\sqrt{\log n}}.\]
    
    We now define the edges of $H_0$. We add every edge between $V_3$ and $V_2$, oriented towards $V_2$. We partition the vertices of $V_3$ into sets of size $s$ that we call \emph{groups}. The number of sets consisting of a vertex from $V_2$ and a group from $V_3$ is $y|V_3|/s=|V_1|/(s-1)$. Consequently, we can assign each vertex $v\in V_1$ to one such set so that every set is assigned exactly $s-1$ times. Fix such an assignment.
    
    For each $v\in V_1$, we add the edge from the vertex assigned to $v$ in $V_2$ oriented towards $v$ and $s$ edges oriented from $v$ to all vertices of the group in $V_3$ assigned to $v$. Thus, we obtain a copy of $D_s$ on these $s+2$ vertices, one copy for each $v\in V_1$. Notice that each vertex in $V_1$ has indegree $1$ and outdegree $s$. Moreover, the complete bipartite graph between $V_2$ and $V_3$ is oriented towards $V_2$, and each of its edges belongs to exactly $s-1$ of the copies of $D_s$ defined above.

    The sets $V_1$ and $V_2$ each remain independent in $H$. It remains to define the edges inside $V_3$. As in the proofs of Lemma \ref{ssat_nlogn} and Lemma \ref{sat_nlogn}, we construct a graph on $V_3$ that saturates $D_s$ while avoiding $Q_s$.

    To this end, let $z=2s\lceil \log |V_3|\rceil$. Let $V_3^{b}$ be the set formed by $z/s$ groups of $V_3$, and let $V_3^{a} = V_3 \setminus V_3^b$.
    Now define $\lceil \log |V_3|\rceil=z/(2s)$ partitions of $V_3^a$ into two sets such that every pair of vertices in $V_3^a$ is separated by at least one of these partitions. For each bipartition $A\cup B=V_3^a$, assign $2$ groups of $V_3^b$ such that no vertex is assigned twice. Connect all vertices of $A$ with the first group of vertices oriented towards the first group. Similarly, connect all vertices of $B$ with the second group oriented towards the second group of vertices. Also, add all edges from the first group towards $B$ and from the second group towards $A$. This forms a complete bipartite graph between $V_3^a$ and the $2s$ vertices of $B$ in these two groups. Repeating this construction for every bipartition, we get a complete bipartite graph with parts $V_3^a$ and $V_3^b$. See Figure \ref{fig:diamondplus}. 

    We added $O(n)$ edges between $V_2$ and $V_3$, $O(n)$ edges between $V_1$ and $V_2\cup V_3$, and \[|V_3^a||V_3^b|=O(|V_3|\log |V_3|)=O(n\sqrt{\log n})\] edges within $V_3$, since $|V_3|\le ns/\sqrt{\log n}$. Thus, $H_0$ has $O(n\sqrt{\log n})$ edges in total.

    Next, we show that $H_0$ avoids $Q_s$. Assume the contrary, that $H_0$ contains a copy of $Q_s$. Observe that $Q_s$ contains $s$ oriented triangles. The subgraph induced by $V_3$ is bipartite, and all edges between $V_2$ and $V_3$ are oriented towards the independent set $V_2$. Hence, the subgraph induced by $V_2\cup V_3$ contains no oriented triangles.
    
    Therefore, some vertex $v\in V_1$ is in this copy of $Q_s$, and thus the oriented triangle through $v$ must have its other two vertices in $V_2$ and $V_3$, respectively. Since every edge between $V_2$ and $V_3$ is contained in exactly $s-1$ oriented triangles, the edge of $Q_s$ that lies in $s$ oriented triangles must be incident to $v$. Yet, the edge from $v$ to $V_3$ has only one oriented triangle through it, so the edge from $V_2$ to $v$ must be the edge of $Q_s$ contained in two oriented triangles. Notice that this edge has exactly $s$ oriented triangles through it, so the $D_s$ part of this $Q_s$ must be exactly the $D_s$ we put on $v$ in this construction phase. But then $v$ must have indegree at least $2$, contradicting the fact that $v$ has indegree exactly $1$. 

    Finally, we prove that almost all non-edges are saturated in $H_0$. We distinguish three possible cases: 
    
    \begin{enumerate}
    \item Suppose that we add a new edge oriented towards a vertex $v\in V_1$. Together with the copy of $D_s$ associated with $v$ and its assigned $(s+1)$-set, this creates a copy of $Q_s$.
    
    \item Suppose that we add a new edge $\orarrow{v_1v_3}$, where $v_1\in V_1$ and $v_3\in V_3$. There is a unique edge from some $v_2\in V_2$ such that $\orarrow{v_2v_1}\in E(H_0)$. Moreover, there exist $s-1$ vertices $v_1^1,v_1^2,\dots v_1^{s-1}\in V_1$ such that their assigned $s+1$-set is the same and contains the vertices $v_2$ and $v_3$. Thus, there are $s-1$ oriented triangles in $H_1$ through $\orarrow{v_2v_3}$. The added edge creates an additional oriented triangle on $v_1,v_3,v_2$, and we obtain a copy of $D_s$. With an appropriate $v_3'\in V_3$, $v_3'\ne v_3$ that has an edge oriented towards $v_2$, we get a copy of $Q_s$.
    
    \item Suppose that we add a new edge $\orarrow{ab}$ with $a,b\in V_3^a$. Choose one of the bipartitions $A\cup B=V_3^a$ constructed above that separates $a$ and $b$. There are two groups associated with this bipartition, one of them together with $a$ and $b$, forms a copy of $D_s$. Taking an additional vertex from $V_1$ that has an edge oriented towards $b$, we obtain a copy of $Q_s$.
    \end{enumerate}

    Thus, the only non-edges that may remain unsaturated are
    $O(n\sqrt{\log n})$ non-edges incident to $X$,
    $O(n\sqrt{\log n})$ non-edges between $V_1$ and $V_2$, $O(\log^2 n)$ non-edges within $V_2$, and $O(\log^2 n)$ non-edges within $V_3^b$.
    In total, this is $O(n\sqrt{\log n})$. We can then greedily add edges to get $H$, finishing the proof.
\end{proof}

\section{Further questions}

We conclude with several remarks and open questions regarding the oriented saturation problem.

\begin{ques}
    Is there a nontrivial universal upper bound for the saturation function? 
\end{ques}

One possible approach is to adapt the construction provided in Lemma \ref{ssat_n}, but only on two isolated vertices (instead of a linear number of them). If that construction avoids the graph $G$, then we can obtain $\binom{n}{2} - \Theta(n)$ as an upper bound for this family of graphs.

\begin{ques} \label{distance_graph}
    It might be tempting to think about Lemma \ref{ver_trans} in the opposite direction. We take any connected graph, and we want to complete it into a vertex-transitive graph. The main question that arises is: how small (in terms of the number of vertices) can such a vertex-transitive graph be? 
\end{ques}

A possible approach is to construct the smallest cyclically-ordered graph for which no two edges with the same distance have different orientations. 

Another question that has also remained open in the edge-ordered setting~\cite{boskesz026} is the following.

\begin{ques}
    Are there finitely many different orders of magnitude for the saturation functions of oriented graphs?
\end{ques}

In Subsection \ref{with_isolated_edge}, we studied the oriented graphs that contain an isolated edge and proved that there is a dichotomy in this case. We also showed various examples with bounded and linear saturation functions. 

\begin{ques}
    Can the saturation functions of oriented graphs containing an isolated edge be characterized?
\end{ques}

While investigating this special case, we introduced the notion of deletion-homogeneous graphs that are useful for producing more examples with a bounded saturation function. It would be interesting to further study this family of oriented graphs, possibly independently of the saturation problem.

We noticed that there are certain similarities, but also differences in studying the saturation in oriented graphs compared to edge-ordered graphs~\cite{boskesz026}. Therefore, it is natural to ask the following.

\begin{ques}
    Can any of the proof techniques developed in this paper be used to study further the saturation functions of edge-ordered graphs? How about the other way around?
\end{ques}

\section*{Acknowledgements}

V.B.~is supported by the European Research Council (ERC) via grant agreement ``ERMiD'' No.\ 101054936, and by the Hungarian National Research, Development and Innovation Office, Advanced grant 153378. 

B.K.~is supported by the European Research Council (ERC) via grant agreement ``ERMiD'' No.\ 101054936, and by the Hungarian National Research, Development and Innovation Office, EXCELLENCE-24 project No.~151504.

\bibliographystyle{plainurl}
\bibliography{orientedsat}

\end{document}